\documentclass{amsart}
\usepackage{amssymb}
\usepackage{amsmath}
\usepackage{hyperref}

\newtheorem{theorem}{Theorem}[section]
\newtheorem{lemma}[theorem]{Lemma}

\newtheorem{corollary}[theorem]{Corollary}

\theoremstyle{definition}

\DeclareMathOperator{\ord}{ord}
\DeclareMathOperator{\supp}{supp}

\begin{document}
\title{Inverse zero-sum problems for certain groups of rank three}
\author{Benjamin Girard} 
\address{Sorbonne Universit\'e, Universit\'e Paris Diderot, CNRS, Institut de Math\'ematiques de Jussieu - Paris Rive Gauche, IMJ-PRG, F-75005, Paris, France}
\email{\texttt{benjamin.girard@imj-prg.fr}}
\author{Wolfgang A.~Schmid}
\address{Universit\'e Paris 13, Sorbonne Paris Cit\'e, LAGA, CNRS, UMR 7539, Universit\'e Paris 8, 
F-93430, Villetaneuse, France, and 
Laboratoire Analyse, G\'eom\'etrie et Applications (LAGA, UMR 7539), COMUE Universit\'e Paris Lumi\`eres, Universit\'e Paris 8, 
CNRS, 93526 Saint-Denis cedex, France}
\email{\texttt{schmid@math.univ-paris13.fr}}


\begin{abstract}
The inverse problem associated to the Erd\H{o}s--Ginzburg--Ziv constant and the $\eta$-constant is solved for finite abelian groups of the form 
$C_2 \oplus C_2 \oplus C_{2n}$ where $n \ge 2$ is an integer.  
\end{abstract}
\maketitle

\section{Introduction}

For $(G,+,0)$ a finite abelian group with exponent $\exp(G)$ the Erd\H{o}s--Ginzburg--Ziv constant, denoted by $\mathsf{s}(G)$, is defined as the smallest nonnegative integer $\ell$ such that each sequence over $G$ of length at least $\ell$ has a subsequence of length $\exp(G)$ whose sum is $0$. The $\eta$-constant, denoted by $\eta(G)$, is defined in the same way except that the length of the subsequence is at most $\exp(G)$ and at least $1$.

These constants have been studied since the early 1960s; for an overview we refer to the survey article \cite{GaoGero06}, in particular to Sections 6 and 7. Their exact values are only known for groups of rank at most two (see \cite[Theorem 6.5]{GaoGero06} and \cite{Reiher_kemnitz,SavChen_kemnitz} for key-contributions), and for a few special types of other groups.  On the one hand, the approach used for groups of rank two can be adapted to apply  to some other groups (see \cite{Luo17,RoyThangadurai18}). On the other hand, there are various results for homocyclic groups, that is, groups of the form $C_n^r$ where $C_n$ denotes a cyclic group of order $n$ and groups that resemble (we refer to \cite[Theorem B]{FanGaoZhong11} for an overview). There is also considerable work on bounds; we refer to \cite{FoxSauermann} and the references therein  for recent results on upper bounds, and we refer to \cite{Edel08} for lower bounds. 

Only recently the exact values were determined for groups of the form $C_2 \oplus C_2 \oplus C_{2n}$ (see  \cite[Theorem 1.3]{FanGaoPengWangZhong13} and \cite[Theorem 1.2.1]{FanZhong16}); this was further extended to groups of the form  $C_2 \oplus C_{2m} \oplus C_{2mn}$ in \cite{GirardSchmid1}. 
Moreover,  for $C_2 \oplus C_2 \oplus C_2 \oplus C_{2n}$ the value of the $\eta$-constant is  known for all $n \ge 1$, and the value of the Erd{\H o}s--Ginzburg--Ziv constant is known for $n\ge 36$ (see \cite[Theorem 1.2.2]{FanZhong16}). 

The associated inverse problems consist in determining all sequences of length  $\mathsf{s}(G)-1$ (respectively $\eta(G)-1$) that have no  zero-sum subsequence of length $\exp(G)$ (respectively of length at most $\exp(G)$ and at least $1$). 
For $\eta(G)$ the inverse problem is solved for groups of rank at most two. For $\mathsf{s}(G)$ the inverse problem is solved for cyclic groups, and for groups 
of rank at most two there is a well-supported conjecture and partial results towards this conjecture are known (see \cite{Wolfgang}). For recent results for certain nonabelian groups see \cite{OhZhong}.  

In the current paper, we solve the inverse problems associated to $\eta(G)$ and to $\mathsf{s}(G)$ for groups of the form $C_2 \oplus C_2 \oplus C_{2n}$.
To solve the inverse problem associated to $\eta(G)$, which is done in Theorem \ref{Inverse_eta}, we use that the structure of minimal zero-sum sequences of maximal length is known for this type of groups (see Theorem \ref{Inverse_Dav}). We recall that a similar approach was used to solve the inverse problem associated to $\eta(G)$ for groups of rank two (see \cite[Theorem 10.7]{GaoGeroldingerPM} or \cite[Section 11.3]{reiherB}). Then, the result for the $\eta$-constant is used to solve the problem for the Erd\H{o}s--Ginzburg--Ziv constant (see Theorem \ref{Inverse_EGZ}). Other tools used in the proof are well-known results on the inverse problem for cyclic groups that we recall in Section \ref{sec_cyclic} as well as ideas used in the proofs of the corresponding direct results mentioned above. 

Conjecturally, there is a tight link between these two constants, namely Gao conjectured (see \cite[Conjecture 6.5]{GaoGero06})  that $\mathsf{s}(G) = \eta(G) + \exp(G)-1$ holds for all finite abelian groups. Furthermore, a close link can also be noted for the structure of extremal sequences in cases where the inverse problem is solved. A close link can also be observed in the current case, yet the exact link is more complex to describe than for groups of rank at most two and for homocyclic groups; we discuss this in more detail after Theorem \ref{Inverse_EGZ}.

\section{Preliminaries}

The notation used in this paper basically matches the one used in \cite{GaoGero06,GeroRuzsa09, GeroKoch06,GrynkiewiczBOOK}. For definiteness, we give a brief summary. 

In this paper,  intervals  are intervals of integers, specifically $[a,b] = \{z \in \mathbb{Z} \colon a \le z \le b \}$.

Let $(G,+, 0)$ be a finite abelian group. For $g$ in $G$, let $\ord(g)$ denote its order in $G$. 
For a subset $A \subset G$, let   $\left\langle A \right\rangle$ denote the subgroup it generates;
 $A$ is called a \emph{generating set} if  $\left\langle A \right\rangle = G$. 
Elements $g_1 , \dots, g_k \in G$ are called \emph{independent} if $\sum_{i=1}^k a_i g_i= 0$, with integers $a_i$, implies that $a_ig_i = 0$ for each $i \in [1,k]$; a set is called  independent when its elements are independent.  We refer to an independent generating set as a basis. 

The \emph{exponent} of $G$  is the least common multiple of the orders of elements of $G$, it is denoted by $\exp(G)$; 
the \emph{rank} of $G$, denoted by $\mathsf{r}(G)$, is the minimum cardinality of a generating subset of $G$. 
For $n$ a positive integer,  $C_n$ denotes a cyclic group of order $n$. 
  
\medskip
By a \emph{sequence} over $G$ we mean an element of the free abelian monoid over $G$, denoted by  $\mathcal{F}(G)$.
We use multiplicative notation for this monoid;  its neutral element is denoted by $1$. 
For a sequence  
\[
S = \displaystyle\prod_{g \in G} g^{v_g}
\]
where   $v_g$ is a nonnegative integer, for each $g\in G$, we denote by: 
\begin{itemize}
\item $\mathsf{v}_g(S) = v_g$ the \emph{multiplicity} of $g$ in $S$.
\item $\mathsf{h}(S)=\max\{\mathsf{v}_g(S) \colon  g \in G\}$ the \emph{height} of $S$.  
\item $\sigma(S)=\sum_{g \in G}\mathsf{v}_g(S)g$ the \emph{sum} of $S$. 
\item $|S| =  \sum_{g \in G}\mathsf{v}_g(S)$ the \emph{length} of $S$. 
\item $\supp(S)=\{g \in G \mid \mathsf{v}_g(S)>0\}$ the \emph{support} of $S$. 
\end{itemize}
Moreover, for $h \in G$, one denotes by $h +S$ the sequence where each element in $S$ is translated by $h$, that is, $h+S = \prod_{g \in G} (h+g)^{v_g}$.

When one writes $S = g_1 \cdots g_{\ell}$ with $g_i \in G$, then the $g_i$ are not necessarily distinct, yet they are determined uniquely up to ordering. A sequence for which the $g_i$ are pairwise distinct, equivalently $\mathsf{v}_g(S) \le 1$ for each $g \in G$, is called \emph{squarefree}.  
A nonempty sequence over $G$ of length at most $\exp(G)$ is called \emph{short}.  

Let $S$ be a sequence  over $G$. A divisor $T$ of $S$ in  $\mathcal{F}(G)$ is called a \emph{subsequence} of $S$; 
by $ST^{-1}$ we denote the sequence such that $T (ST^{-1}) = S$.
Subsequences $T_1, \dots, T_k$  of $S$  are called \emph{disjoint} if the product 
$T_1 \cdots T_k$ is also a subsequence of $S$.

An element $s \in G$ is a \emph{subsum} of $S$ if
\(
s=\sigma(T) \text{ for some } 1 \neq T  \mid S.
\)
The sequence $S$ is called a \emph{zero-sum free sequence} if $0$ is not a subsum. If $\sigma(S)=0$, then $S$ is called a \emph{zero-sum sequence}; 
if, in addition, one has $\sigma(T) \neq 0$ for all proper and nonempty subsequences $T $ of $S$, then $S$ is called a \emph{minimal zero-sum sequence}.

For $L $ a set of integers, we set
\[
\Sigma_L(S) = \{\sigma(T) \colon 1 \neq T \mid S\text{ with } |T| \in L\}.
\] 
For $\Sigma_{\mathbb{Z}_{\ge 1}}(S)$ we just write $\Sigma(S)$, and for $\Sigma_{\{k\}}(S)$ we write $\Sigma_k(S)$. 

Moreover, $\mathsf{s}_L(G)$ is the smallest nonnegative integer $\ell$ (if one exists)  such that for each sequence $S \in \mathcal{F}(G)$ with $|S|\ge \ell$, one has $0  \in \Sigma_{L}(S)$; if no such integer exists, then one sets $\mathsf{s}_L(G) = \infty$. The constants we mentioned in the introduction are natural special cases; the choice of $\exp(G)$ in the definitions is natural, for example it is the smallest integer for which the respective constants are finite. 

\begin{itemize}
\item The \emph{Davenport constant}, denoted by $\mathsf{D}(G)$,  is $\mathsf{s}_{\mathbb{Z}_{\ge 1}}(G)$.   

\item The \emph{Erd\H{o}s--Ginzburg--Ziv constant}, denoted by $\mathsf{s}(G)$,  is $\mathsf{s}_{\exp(G)}(G)$.   

\item The \emph{$\eta$-constant}, denoted by $\eta(G)$,  is $\mathsf{s}_{[1,\exp(G)]}(G)$.
\end{itemize}
The Davenport constant of $G$ is also equal to the maximal length of a minimal zero-sum sequence over $G$.

\section{Results for cyclic groups and  auxiliary results}
\label{sec_cyclic}

The $\eta(G)$ and $\mathsf{s}(G)$ invariants for cyclic groups are well-known and the inverse problems are solved. Indeed, in the case of cyclic groups also the structure of sequences considerably shorter than $\eta(G)-1$ and $\mathsf{s}(G)-1$
without the respective zero-sum subsequences is known (for example, see \cite{Gaoetal, SavChen07bis, SavChen07, Yuan}). We only recall a special case we need for our proofs. While the result below is formulated for \emph{short} zero-sum subsequences, since this fits the current context, having no short zero-sum subsequence and having no nonempty zero-sum subsequence are equivalent for \emph{cyclic} groups, and the result is usually phrased in the latter form.  

The following result is a direct consequence of \cite[Theorem 8]{SavChen07bis} (also see \cite[Theorem 11.1]{GrynkiewiczBOOK}).   

\begin{theorem}
\label{Inverse_eta_cyc} 
Let $H \simeq C_{n}$ where $n \ge 3$ is an integer. 
\begin{enumerate}
\item Every sequence $T$ of length $\left|T\right|=\eta(H)-1=n-1$ not containing any short zero-sum sequence has the form
\[
T=g^{n-1} \text{ where } \ord(g)=n.\]
\item
Every sequence $S$ of length $\left|T\right|=\eta(H)-2=n-2$ not containing any short zero-sum sequence  has the form:
\begin{enumerate}
\item \(
T=g^{n-2} \text{ where } \ord(g)=n,\) or 
\item \(T=g^{n-3}(2g) \text{ where } \ord(g)=n.\)
\end{enumerate}
In the former case $ \Sigma(T)  = H \setminus \{0,-g\}$, in the latter case  $ \Sigma(T)  = H \setminus \{0\} $ (except for $n=3$ where the latter case coincides with the former).
\end{enumerate}
\end{theorem}

For the following results see \cite{SavChen07}, especially the discussion after Corollary 7 there.  

\begin{theorem}
\label{Inverse_EGZ_cyc} 
Let $H\simeq C_{n}$ where $n \ge 3$ is an integer.
\begin{enumerate}
\item Every sequence $T$ of length $\left|T\right|= \mathsf{s}(H)-1=2n-2$ not containing any short zero-sum sequence has the form
\(T=g^{n-1}h^{n-1}\)  where $g,h \in H$ and $ \ord(g-h)=n$.
\item Every sequence $S$ of length $\left|T\right|=\mathsf{s}(H)-2=2n-3$ not containing any short zero-sum sequence  has the form:
\begin{enumerate}
\item \(T=g^{n-1}h^{n-2} \text{ where } \ord(g-h)=n\) or 
\item \(T=g^{n-1}h^{n-3}(2h-g) \text{ where } \ord(g-h)=n.\)
\end{enumerate}
In the former case $ \Sigma_{n-2}(T)  = H \setminus \{-g-h\}$, in the latter case  $ \Sigma_{n-2}(T)  = H$ (except for $n=3$ where the latter case coincides with the former).
\end{enumerate}
\end{theorem}

We give two results on groups of the form $C_2^r$ that are needed for the proofs of our main results.  
For the following well-known result see, e.g., \cite[Theorem 7.2]{FreezeSchmid10}. 

\begin{lemma}\label{sum_le3} For $r \ge 2$, one has 
$\mathsf{s}_{[1,3]}(C_2^r)  = 1 +2^{r-1}$. 
\end{lemma}

We also need a result on zero-sum sequences of length $4$. 

\begin{lemma}\label{sum_=4}
Let $S$ be a squarefree sequence of length $5$ over $C_2^3$. Then $S$ has a unique 
zero-sum subsequence of length $4$. 
\end{lemma}
\begin{proof}
Since the property is invariant under translation of the sequence, we can assume that $0 \mid S$. 
The four nonzero elements in $S$ cannot be contained in a proper subgroup, 
and thus $S$ contains three independent elements $f_1, f_2, f_3$. 

Let $g$ denote the fourth nonzero element. Either $g = f_1 + f_2 + f_3$ and the sequence 
$gf_1f_2f_3$ is the unique zero-sum subsequence of length four, or $g$ is the sum of two of the independent elements, 
say $g= f_1+f_2$, and $0gf_1f_2$ is the unique  zero-sum subsequence of length four.
\end{proof}

\section{Inverse problem associated to $\eta\left(C_2 \oplus C_2 \oplus C_{2n}\right)$ for $n \geq 2$}

We determine the structure of all sequences over $C_2 \oplus C_2 \oplus C_{2n}$, for $n\ge 2$, of length $\eta\left(C_2 \oplus C_2 \oplus C_{2n}\right)-1=2n+3$ that do not have a short zero-sum subsequence. The case $n=1$, that is, $C_2^3$, is different yet well-known and direct. For completeness we recall that $\eta(C_2^3)-1 =7$, and the only example of a sequence of length $7$ without zero-sum subsequence of length $2$ is the squarefree sequence of all nonzero elements; in fact, this is true for any group of the form $C_2^r$, because the only zero-sum sequence of length at most $\exp(G)$ are $0$ and $g^2$ for $g \in C_2^r$.

\begin{theorem}\label{Inverse_eta} 
Let $G \simeq C_2 \oplus C_2 \oplus C_{2n}$, where $n \ge 2$. 
A sequence $S$ over $G$ of length $|S|=\eta(G)-1=2n+3$ contains no short zero-sum subsequence if and only if there exists a basis $\{f_1,f_2, f_3\}$ of $G$, where $\ord(f_1) = \ord(f_2) = 2$ and $\ord(f_3) = 2n$, such that $S$ is equal to one of the following sequences:
\begin{enumerate}
\item[$(\eta 1)$] 
$f^{2n-1-2v}_3 (f_3+f_2)^{2v+1}f_2(af_3+f_1)\bigl((1-a)f_3+f_2+f_1\bigr)$ with $v \in [ 0, n-1 ]$ and $a \in [ 2, n-1 ]$.
\item[$(\eta 2)$] 
$f^{2n-1}_3 (af_3 + f_2)\bigl((1 - a)f_3 + f_2\bigl)(bf_3 + f_1)\bigl((1 - b)f_3 + f_1\bigl)$ with $a, b \in [ 2, n - 1 ]$ and $a \ge b$.
\item[$(\eta 3)$] 
$\prod^{2n+1}_{i=1}(f_3 + d_i)f_2f_1$ with $S' = \prod^{2n+1}_{i=1} d_i \in \mathcal{F}(\langle f_1, f_2 \rangle)$ and $\sigma(S') \notin \supp(S')$.
\end{enumerate}
\end{theorem}
For $n=2$, in fact only the last case can occur as $[2,n-1]$ is empty.
To prove this result we use that the structure of all minimal zero-sum subsequences of maximal length is known (see  \cite[Theorem 3.13]{Schmid11}).

\begin{theorem}\label{Inverse_Dav}
Let $G \simeq C_2 \oplus C_2 \oplus C_{2n}$, where $n \ge 1$. 
A sequence $S$ over $G$ is a minimal zero-sum sequence of length $|S|=\mathsf{D}(G)$ if and only if there exists a basis $\{f_1,f_2, f_3\}$ of $G$, where
$\ord(f_1) = \ord(f_2) = 2$ and $\ord(f_3) = 2n$, such that $S$ is equal to one of the following sequences:
\begin{enumerate}
\item[$(\mathsf{D}1)$] 
$f^{v_3}_3 (f_3+f_2)^{v_2}(f_3+f_1)^{v_1}(-f_3+f_2+f_1)$ with $v_i \in \mathbb{N}$ odd, $v_3 \ge v_2 \ge v_1$ and $v_3 + v_2 + v_1 = 2n + 1$.
\item[$(\mathsf{D}2)$] 
$f^{v_3}_3 (f_3+f_2)^{v_2}(af_3+f_1)(-af_3+f_2+f_1)$ with $v_2, v_3 \in \mathbb{N}$ odd $v_3 \ge v_2$ and $v_2 + v_3 = 2n$ and $a \in [ 2, n - 1 ]$.
\item[$(\mathsf{D}3)$] 
$f^{2n-1}_3 (af_3 + f_2)(bf_3 + f_1)(cf_3 + f_2 + f_1)$ with $a + b + c = 2n + 1$ where $a \le b \le c$, and $a, b \in [ 2, n - 1]$, $c \in [ 2, 2n - 3 ] \setminus \{n, n + 1\}$.
\item[$(\mathsf{D}4)$] 
$f^{2n-1-2v}_3 (f_3+f_2)^{2v}f_2(af_3+f_1)\bigl((1-a)f_3+f_2+f_1\bigr)$ with $v \in [ 0, n-1 ]$ and $a \in [ 2, n-1 ]$.
\item[$(\mathsf{D}5)$] 
$f^{2n-2}_3 (af_3 + f_2)\bigl((1 - a)f_3 + f_2\bigr)(bf_3 + f_1)\bigl((1 - b)f_3 + f_1\bigr)$ with $a \ge b$ and $a, b \in [ 2, n - 1 ]$.
\item[$(\mathsf{D}6)$] 
$\prod^{2n}_{i=1}(f_3 + d_i)f_2f_1$ where $S' = \prod^{2n}_{i=1} d_i \in \mathcal{F}(\langle f_1, f_2 \rangle)$ with $\sigma(S') = f_1 + f_2$.
\end{enumerate}
\end{theorem}
As above, for $n\le 2$, only some of these cases occur, namely  the first and the last. 
We now prove our result. 

\begin{proof}[Proof of Theorem \ref{Inverse_eta}]
Let $G \simeq C_2 \oplus C_2 \oplus C_{2n}$ where $n \geq 2$, and  let $H \simeq C_n$ be the subgroup of $G$ such that $G/H \simeq C_2^3$. Note that there is indeed a unique such subgroup $H$. 
We use the inductive method with 
\[
H \hookrightarrow G \overset{\pi} \rightarrow G/H.
\]
Let $S$ be a sequence over $G$ such that $\left|S\right| = \eta(G)-1=2n+3$ that contains no short zero-sum subsequence.

\medskip
First, let us prove that $S$ contains no element $g$ in $H$. Indeed, if it were the case, setting  $g=T_0$, we have $ST_0^{-1}$ whose length is $|ST_0^{-1}|=2n+2 = \eta(G/H)+2(n-3)$ would contain $n-2$ disjoint nonempty subsequences $T_1,\dots,T_{n-2} $ such that $\left|T_i\right| \leq 2$ and $\sigma\left(T_i\right) \in H$ for all $i \in [1,n-2]$.  In particular, $S = T_0T_1\cdots T_{n-2}S_0$ where $S_0$ has length $|S_0|\ge  6$. 

Now, note that all elements of $\pi(S_0)$ must be nonzero and distinct, otherwise we could extract from $S_0$ yet another subsequence $T_{n-1} \mid S_0$ satisfying $|T_{n-1}| \le 2$ and $\sigma\left( T_{n-1} \right) \in H$. 
Then,  we have $T=\prod^{n-1}_{i=0} \sigma(T_i)$ is a sequence over $H$ of length $|T|=n$, so it contains a nonempty zero-sum subsequence, say $\sum_{i\in I} \sigma(T_i) = 0$ with $\emptyset \neq I \subset [0,n-1]$. Thus, $\prod_{i \in I}T_i$ is a subsequence of $S$ with sum $0$ and length at most $2 |I|\le 2n$.  

However, a squarefree sequence of at least $6$ nonzero elements over $G/H \simeq C_2^3$ has a zero-sum subsequence of length at most $3$ (see Lemma \ref{sum_le3}).  Applying this to $\pi(S_0)$ gives that $S_0$  contains a subsequence $T_{n-1}$ of length $|T_{n-1}| =  3$ such that $\sigma\left(T_{n-1}\right) \in H$. Then, again $T=\prod^{n-1}_{i=0} \sigma(T_i)$ is a sequence over $H$ of length $|T|=n$. It contains a nonempty zero-sum subsequence, and whence $\prod^{n-1}_{i=0} T_i$ contains a nonempty zero-sum subsequence. Since  $|\prod^{n-1}_{i=0} T_i|\le 2n$, it follows that this is a short zero-sum subsequence. Since $\prod^{n-1}_{i=0} T_i$ is a subsequence of $S$, we see that $S$ contains a short zero-sum subsequence, which is a contradiction.

\medskip
Thus, we know that $S$ contains no element in $H$.  
Similarly as above, since $\left|S\right| = 2(n-3) + 9 \ge  \eta(G/H)+2(n-3)$, there exist $n-2$ disjoint subsequences $T_1,\dots,T_{n-2}$ of $ S$ such that $\sigma\left(T_i\right) \in H$ and $\left|T_{i}\right| \le 2$ for all $i \in [1,n-2]$; yet, since $S$ contains no element from $H$ in fact  $\left|T_{i}\right| = 2$  for all $i \in [1,n-2]$. In particular, $S = T_1\cdots T_{n-2}S_0$ where $S_0$ has length $|S_0|=7$. 

\medskip
We now try to obtain further information on both $S_0$ and the sequences $T_i$ for  $i \in  [1, n-2]$.
On the one hand,  at most one of the elements in $\pi(S_0)$  has multiplicity at least $2$ and none has multiplicity at least $4$. To see this it suffices to note that  otherwise we could extract from $S_0$ two disjoint subsequences $T_{n-1}$ and $T_n$ of length $2$ whose sums are in $H$. Then, arguing as above, $S$ would contain a short zero-sum subsequence, a contradiction. In particular, we have $|\supp(\pi(S_0))| \ge 5$.

On the other hand, $T=\prod^{n-2}_{i=1} \sigma(T_i)$ is a sequence over $H$ of length $|T|=n-2$, containing no nonempty zero-sum subsequence, otherwise $S$ would contain a short zero-sum subsequence, a contradiction.

By Theorem \ref{Inverse_eta_cyc} we get  $|H \setminus -\left( \Sigma\left(T\right) \cup \{0\}\right)|\le 1$. 
 Since $|\supp(\pi(S_0))| \ge 5$, one can extract from $S_0$ a subsequence $U \mid S_0$ such that $\left|U\right| \le 4$ and $\sigma\left(U\right) \in H$ (see Lemma \ref{sum_le3}); in fact, we could even find a sequence of length at most $3$. Since $S$ contains no short zero-sum subsequence, it follows that $\sigma\left(U\right) \in H \setminus -\left(\Sigma\left(T\right) \cup \{0\}\right)$.

\medskip
For $n=2$, this directly yields that $\sigma(U)=g$ where $g$ is the nonzero element  of $H$.
For $n=3$, we first infer that $H \setminus -\left(\Sigma\left(T\right) \cup \{0\}\right) \neq \emptyset$, which by Theorem \ref{Inverse_eta_cyc} implies that $T = g^{n-2}$ for some $g\in H$ with $\ord(g)=n$, and $\sigma(U)=g$.  
It follows that  there exists an element $g \in H$ satisfying $\ord(g)=n$ such that $S$ can be decomposed as $T_1 \cdots T_{n-2} S_0$ where $\sigma(T_i)=g$ for all $i \in [1,n-2]$, and for each $U \mid S_0$ such that $\left|U\right| \le 4$ and $\sigma\left(U\right) \in H$ one has $\sigma(U)= g$. 

\medskip
\textbf{Case 1.} $\pi(S_0)$ is squarefree.
Then, $S_0$ itself is squarefree and $\supp(\pi(S_0)) =G/H \setminus  \{0\}$.
We now fix a basis $\{e_1, e_2, e_3 \}$ of $G/H$; this quotient group is isomorphic to $C_2^3$. 
For every set $\emptyset \neq I \subset \{1,2,3\}$, we write $e_I=\sum_{i\in I}e_i$ and we denote by $a_I$ the \emph{unique} element of $S_0$ such that $\pi(a_I)=e_I$.  We have $U_0=a_{\{1,2\}}a_{\{2,3\}}a_{\{1,3\}}$, $U_k=a_ia_ja_{\{i,j\}}$ and $V_i=a_{\{i, j\}}a_{\{i, k\}}a_j a_k$ for $\{i ,j, k\}=\{1,2,3\}$ are subsequences of $S_0$ satisfying $|U_k|=3$ and $\sigma(U_k) \in H$ as well as $|V_i|=4$ and $\sigma(V_i) \in H$.
In addition, we have $U_0 U_1 U_2U_3=V_1 V_2 V_3$.

\medskip
By the argument above it follows that $\sigma(U_k)=\sigma(V_i)=g$ for all $i,k \in [1,3]$, which yields $4g=\sigma(U_0U_1U_2U_3)=\sigma(V_1V_2V_3)=3g$. It follows that $g=0$, which is a contradiction.

\medskip
\textbf{Case 2.} $\pi(S_0)$ is not squarefree. We recall that there is a unique element with multiplicity at least $2$. 
In this case, $S_0$ contains a subsequence $T_{n-1} \mid S_0$ such that $\left|T_{n-1}\right| = 2$ and $\sigma\left(T_{n-1}\right) \in H$.
We get again $\sigma(T_{n-1}) = g$. 
Thus  $S = T_1\cdots T_{n-1}S'_0$ where $\sigma\left(T_i\right)=g$ for all $i \in [1,n-1]$ and $S'_0$ has length $|S'_0|=5$.
We know that $\pi(S'_0)$ is squarefree and does not contain $0$.

\medskip
By Lemma \ref{sum_=4}, $S'_0$ contains a unique subsequence $U=stuv \mid S'_0$ such that $\left|U\right| = 4$ and $\sigma\left(U\right) \in H$. 
First, we have $\sigma(U)=s+t+u+v=g$.
Now, setting $S'_0=Uw$, and since $\pi(s)+\pi(t),\pi(s)+\pi(u),\pi(s)+\pi(v)$ are nonzero and pairwise distinct elements of the set $G/H \setminus \{0,\pi(s),\pi(t),\pi(u),\pi(v)\}$ containing $\pi(w)$, exactly one of them, say $\pi(s)+\pi(t)$, is equal to $\pi(w)$. Thus $\pi(w)=\pi(s)+\pi(t)=\pi(u)+\pi(v)$.
Now, $U_1=wst$ and $U_2=wuv$ are two subsequences of $S'_0$ satisfying $|U_1|=|U_2|=3$ and $\sigma(U_1),\sigma(U_2) \in H$.
This yields $\sigma(U_1)=w+s+t=g$ and $\sigma(U_2)=w+u+v=g$, so that $s+t=u+v$. Since also $s+t+u+v=g$, it follows that $2w = g$ and $w = s+t=u+v$.

\medskip
In addition, $V=T_1\cdots T_{n-1}U$ is a zero-sum subsequence of $S$ of length $|V|=2(n-1)+4=2n+2=\mathsf{D}(G)$. Since $S$ contains no short zero-sum subsequence, it cannot be decomposed into two zero-sum subsequences, as one of them would be short. That is, $V$ is a minimal zero-sum sequence over $G$ of length $|V|=\mathsf{D}(G)$. 

\medskip
Let us now give a summary of everything we proved so far.
For every sequence $S$ over $G$ such that $\left|S\right| = 2n+3$ and containing no short zero-sum subsequence, there exists an element $g \in H$ with $\ord(g)=n$ such that $S$ can be decomposed as $T_1\cdots T_{n-1}S'_0$ where $T_1,\dots,T_{n-1}$ are subsequences of $S$  that satisfy $\left|T_i\right| = 2$ and $\sigma\left(T_i\right)=g$ for all $i \in [1,n-1]$, and where $S'_0$ is a sequence of length $|S'_0|=5$.
Moreover, in every such decomposition of $S$, the sequence $\pi(S'_0)$ is squarefree and does not contain $0$, more precisely the sequence $S_0'$ is squarefree and can be uniquely decomposed as $S'_0=Uw$ where $U \mid S'_0$ satisfies $|U|=4$ and $\sigma(U) \in H$. 
Also, we can write $U=stuv$ so that the equalities $s+t=u+v=w$ and $\sigma(U)=2w=g$ hold. 
Finally, $V=T_1\cdots T_{n-1}U$ is a minimal zero-sum sequence over $G$ of length $|V|=2n+2=\mathsf{D}(G)$.
 
\medskip
In particular, using Theorem \ref{Inverse_Dav}, we know there exists a basis $\{f_1,f_2, f_3\}$ of $G$, where $\ord(f_1) = \ord(f_2) = 2$ and $\ord(f_3) = 2n$, such that $V$ has one out of six possible forms. We treat each of these cases separately; we recall that for $n=2$ only the first and the last can occur, so that for the others we can assume that $n \ge 3$. 

\medskip
If $V$ is of type $(\mathsf{D}1)$, then in order to have a support of size at least five, $S'_0$ must contain one copy of every member of $\supp(V)$. 
Since the four elements of $\supp(V)$ sum up to $2f_3 \in H$, we obtain $U=f_3 (f_3+f_2) (f_3+f_1) (-f_3+f_2+f_1)$, but it is easily seen that $U$ cannot be decomposed as the product of two sequences of length $2$ having the same sum.

\medskip
If $V$ is of type $(\mathsf{D}2)$, then in order to have a support of size at least five, $S'_0$ must contain one copy of every member of $\supp(V)$. 
Since the four elements of $\supp(V)$ sum up to $2f_3 \in H$, we obtain $U=f_3 (f_3+f_2) (af_3+f_1)(-af_3+f_2+f_1)$ with $a \in [ 2, n - 1 ]$. 
In addition, $U$ can be decomposed as the product of two sequences of length $2$ having the same sum only if $1-a=1+a \pmod{2n} $, that is to say only if $a \in \{0,n\}$, which is a contradiction. 

\medskip
If $V$ is of type $(\mathsf{D}3)$, then in order to have a support of size at least five, $S'_0$ must contain one copy of every member of $\supp(V)$. 
Since the four elements of $\supp(V)$ sum up to $2f_3 \in H$, we obtain $U=f_3 (af_3 + f_2)(bf_3 + f_1)(cf_3 + f_2 + f_1)$ with $a + b + c = 2n + 1$ where $a \le b \le c$, and $a, b \in [ 2, n - 1]$, and $c \in [ 2, 2n - 3 ] \setminus \{n, n + 1\}$. 
Now, $U$ can be decomposed as the product of two sequences of length $2$ having the same sum only if $1+a=b+c \pmod{2n}$ or $1+b=a+c  \pmod{2n}$  or $1+c=a+b  \pmod{2n}$; recall that $a+b+c = 1  \pmod{2n}$. This is possible only if $a \in \{0,n\}$ or $b \in \{0,n\}$ or $c \in \{0,n\}$, which yields a contradiction.

\medskip
If $V$ is of type $(\mathsf{D}4)$, then in order to have a support of size at least five, $S'_0$ must contain at least four elements of $\supp(V)$.
First, note that $S'_0$ contains at most one copy of $f_3$ and at most one copy of $f_3+f_2$. In other words, $T_1\cdots T_{n-1}$ contains at least $2n-2v-2$ copies of $f_3$ and at least $2v-1$ copies of $f_3+f_2$.
Assume that there exists no $i \in [1, n-1]$ satisfying $T_i=f^2_3$ or $T_i=(f_3+f_2)^2$.
Since $(2n-2v-2)+(2v-1)=2n-3 \ge n$, there exists $i \in [1,n-1]$ satisfying $S_i=f_3(f_3+f_2)$ so that $g=2f_3+f_2$. Yet, this is not an element of $H$, which is a contradiction. It follows that $T_i=f^2_3$ or $T_i=(f_3+f_2)^2$ for at least one $i \in [1,n-1]$ which yields $g=2f_3$.
In particular, for every $i \in [1,n-1]$ such that $T_i$ contains $(f_3+f_2)$, we have $T_i=(f_3+f_2)^2$ so that $T_1\cdots T_{n-1}$ contains an even number of copies of $f_3+f_2$. 
Therefore, the number of copies of $f_3+f_2$ contained in $S'_0$, which is at most one, must be zero. 
We thus obtain $U=f_3 f_2(af_3+f_1)\bigl((1-a)f_3+f_2+f_1\bigr)$ with $a \in [ 2, n-1 ]$.

\medskip
Since $1+a \neq 1-a$ (mod $2n$) and $1+(1-a) \neq a  \pmod{2n}$, there is only one possible decomposition of $U$ as the product of two sequences of length $2$ having the same sum, which yields $w=f_3+f_2$. Therefore,
\[
S=f^{2n-1-2v}_3 (f_3+f_2)^{2v+1}f_2(af_3+f_1)\bigl((1-a)f_3+f_2+f_1\bigr)
\]
with $v \in [ 0, n-1 ]$ and $a \in [ 2, n-1 ]$. 
It remains to check that this sequence contains no short zero-sum subsequence.
Indeed, since $V$ is a minimal zero-sum sequence over $G$, any short zero-sum subsequence $S' \mid S$ must satisfy $(f_3+f_2)^{2v+1} \mid S'$. 
Now, if $S'$ contains neither $(af_3+f_1)$ nor $((1-a)f_3+f_2+f_1)$, we must have $S'=f^{2n-1-2v}_3(f_3+f_2)^{2v+1}f_2$ so that $|S'|=2n+1 > 2n$ which is a contradiction. If $S'$ contains either $(af_3+f_1)$ or $\bigl((1-a)f_3+f_2+f_1\bigr)$ then it must contain both of them, which yields $S'=f^{2n-2-2v}_3 (f_3+f_2)^{2v+1}(af_3+f_1)\bigl((1-a)f_3+f_2+f_1\bigr)$ so that $|S'|=2n+1 > 2n$, which is a contradiction.

\medskip
If $V$ is of type $(\mathsf{D}5)$, then in order to have a support of size at least five, $S'_0$ must contain at least four elements of $\supp(V)$.
Since $S'_0$ contains at most one copy of $f_3$, the sequence $T_1\cdots T_{n-1}$ contains at least $2n-3 \ge n$ copies of $f_3$.
This implies that every $T_i$ contains at least one copy of $f_3$ and that at most one of them is different from $f^2_3$.
This yields $g=2f_3$ so that $T_i = f^2_3$ for every $i \in [1,n-1]$.
Therefore, the elements $(af_3 + f_2), \bigl((1 - a)f_3 + f_2\bigr), (bf_3 + f_1)$ and $\bigl((1 - b)f_3 + f_1\bigr)$ belong to $S'_0$ and since their sum is equal to $2f_3 \in H$, we have $U=(af_3 + f_2) \bigl((1 - a)f_3 + f_2\bigr) (bf_3 + f_1)\bigl((1 - b)f_3 + f_1\bigr)$ with $a \ge b$ and $a, b \in [ 2, n - 1 ]$. 

\medskip
Now, let $w=\alpha f_3+d$ where $\alpha \in [0,2n-1]$ and $d \in \langle f_1, f_2 \rangle$.
On the one hand, $2w=g=2f_3$ yields $\alpha \in \{1,n+1\}$.
On the other hand, since $w$ must be equal to the sum of two elements of $U$, we have $d \in \{0,f_1+f_2\}$.  

\medskip
In case $d=0$, we get that $\alpha=1$ which yields $w=f_3$ so that
\[
S=f^{2n-1}_3 (af_3 + f_2)\bigl((1 - a)f_3 + f_2\bigr)(bf_3 + f_1)\bigl((1 - b)f_3 + f_1\bigr)
\]
with $a \ge b$ and $a, b \in [ 2, n - 1 ]$. 
It remains to check that this sequence contains no short zero-sum subsequence.
Indeed, since $V$ is a minimal zero-sum sequence over $G$, any short zero-sum subsequence $S' \mid S$ must satisfy $f^{2n-1}_3 \mid S'$, which implies $|S'|=2n$ so that $S'=f^{2n}_3 \nmid S$, a contradiction.

\medskip
In the case $d=f_1+f_2$ and $\alpha=1$, we have $w=f_3+f_2+f_1$. 
Since $a+b=(1-a)+(1-b) \pmod{2n}$ if and only if $a+b=n+1 \neq 1 \pmod{2n}$, we have 
$a+(1-b)=(1-a)+b  \pmod{2n} $  that is to say $a=b$.
We thus obtain
\[
S=f^{2n-2}_3 (af_3 + f_2)\bigl((1 - a)f_3 + f_2\bigr)(af_3 + f_1)\bigl((1 - a)f_3 + f_1\bigr)(f_3+f_2+f_1)
\]
with $a\in [ 2, n - 1 ]$.
Yet, $S'=f^{2n-2a-1}_3 (af_3 + f_2)(af_3 + f_1)(f_3+f_2+f_1) \mid S$ is a nonempty zero-sum subsequence of length $|S'|=2n-2a+2 \le 2n$. 

\medskip
In the case $d=f_1+f_2$ and $\alpha=n+1$, we have $w=(n+1)f_3+f_2+f_1$.
Since $a+(1-b)=(1-a)+b  \pmod{2n}$  if and only if $a=b$ which implies $a+1-b=1 \neq n+1 \pmod{2n}$, we have
$a+b=(1-a)+(1-b)$ (mod $2n$) if and only if $a+b=n+1 \pmod{2n}$ .
We thus obtain
\[S=f^{2n-2}_3 (af_3 + f_2)\bigl((1 - a)f_3 + f_2\bigr)\bigl((n+1-a)f_3 + f_1\bigr)\bigl((a - n)f_3 + f_1\bigr)\bigl((n+1)f_3+f_2+f_1\bigr)\]
with $a \in [ 2, n - 1 ]$.
Yet, $S'=f^{2n-2a-1}_3 (af_3 + f_2)\bigl((a-n)f_3 + f_1\bigr)\bigl((n+1)f_3+f_2+f_1\bigr) \mid S$ is a nonempty zero-sum subsequence of length $|S'|=2n-2a+2 \le 2n$.

\medskip
If $V$ is of type $(\mathsf{D}6)$, we have $V=\prod^{2n}_{i=1}(f_3 + d_i)f_2f_1$ where $\prod^{2n}_{i=1} d_i \in \mathcal{F}(\langle f_1, f_2 \rangle)$ is such that $\sum^{2n}_{i=1} d_i = f_1 + f_2$. 
Then $w=f_3+d$ for some $d \in \langle f_1, f_2 \rangle$ so that we can write $S=\prod^{2n+1}_{i=1}(f_3 + d_i)f_2f_1$ where $S'=\prod^{2n+1}_{i=1} d_i \in \mathcal{F}(\langle f_1, f_2 \rangle)$. 
It is easily seen that such a sequence $S$ contains no short zero-sum subsequence if and only if $S'=\prod^{2n+1}_{i=1} d_i \in \mathcal{F}(\langle f_1, f_2 \rangle)$ contains no zero-sum subsequence of size $2n$, that is to say if and only if $\sigma(S') \notin \supp(S')$.
\end{proof}

\section{Inverse problem associated to $\mathsf{s}\left(C_2 \oplus C_2 \oplus C_{2n}\right)$ for $n \geq 2$}

We turn to the inverse problem associated to $\mathsf{s}\left(C_2 \oplus C_2 \oplus C_{2n}\right)$ for $n \ge 2$. Again, the case $n=1$, that is, $C_2^3$, is different yet well-known and direct; $\mathsf{s}(C_2^3)-1 =8$ and the only example of a sequence of length $8$ without zero-sum subsequence of length $2$ is the squarefree sequence of all elements, because the only zero-sum sequences of length $\exp(G)$ are $g^2$ for $g \in C_2^r$. The proof of this result uses Theorem \ref{Inverse_eta}. 

\begin{theorem}\label{Inverse_EGZ} 
Let $G \simeq C_2 \oplus C_2 \oplus C_{2n}$, where $n \ge 2$. 
A sequence $S$ over $G$ of length $|S|=\mathsf{s}(G)-1=4n+2$ contains no zero-sum subsequence of length $\exp(G)$ if and only if there exist a basis  $\{f_1,f_2, f_3\}$ of $G$, where
$\ord(f_1) = \ord(f_2) = 2$ and $\ord(f_3) = 2n$, and an $f \in G$ such that $-f+S$ is equal to one of the following sequences:
\begin{enumerate}
\item[$(\mathsf{s}1)$] 
$0^{2\alpha+1} f^{2n-2\alpha-1}_2 f^{2n-1-2\beta }_3 (f_3+f_2)^{2\beta +1}(af_3+f_1)\bigl((1-a)f_3+f_2+f_1\bigr)$ with  $a \in [ 2, n-1 ]$ and $\alpha, \beta \in [0,n-1]$.
\item[$(\mathsf{s}2)$] 
$0^{2n-1}f^{2n-1}_3 (af_3 + f_2)\bigl((1 - a)f_3 + f_2\bigr)(bf_3 + f_1)\bigl((1 - b)f_3 + f_1\bigr)$ with $a, b \in [ 2, n - 1 ]$ and $a \ge b$.
\item[$(\mathsf{s}3)$] 
$0^{2\alpha+1} f^{2\beta+1}_1 f^{2\gamma+1}_2 \prod^{2n+1}_{i=1}(f_3 + d_i)$
where $\alpha,\beta,\gamma \in [0,n-1]$ are such that $\alpha+\beta+\gamma=n-1$, $S' = \prod^{2n+1}_{i=1} d_i \in \mathcal{F}(\langle f_1, f_2 \rangle)$ and $\sigma(S') \notin \supp(S')$.
\end{enumerate}
\end{theorem}
As in the result for the $\eta$-constant, for $n=2$ only the last case can occur.
Before giving the proof of this result we put the result in context and derive two corollaries.
The first is about the height of these extremal sequences, that is, the maximal multiplicity of an element in these sequences.  

\begin{corollary}
Let $G \simeq C_2 \oplus C_2 \oplus C_{2n}$, where $n \ge 2$. 
For every sequence $S$ over $G$ of length $|S|=\mathsf{s}(G)-1$ without zero-sum subsequence of length $\exp(G)$ one has
\[
\mathsf{h}(S) \ge 
\begin{cases} 
\frac{2n+3}{3} & \text{ if } n \equiv 0 \pmod{3}\\  
\frac{2n+1}{3} & \text{ if } n \equiv 1 \pmod{3}\\  
\frac{2n+5}{3} & \text{ if } n \equiv 2 \pmod{3}
 \end{cases}\] 
and these bounds are attained. 
\end{corollary}
\begin{proof}
For sequences of the second type in Theorem \ref{Inverse_EGZ}, it is clear that $\mathsf{h}(S)= 2n - 1$, and the claim follows. 
For sequences of the first type in Theorem \ref{Inverse_EGZ}, considering for example $0$ and $f_2$ we see that $\mathsf{h}(S)> 2n/2$; again the claim follows. For sequences of the third type in Theorem \ref{Inverse_EGZ}, considering $0$, $f_1$, and $f_2$ we see that $\mathsf{h}(S)\ge 2\lfloor (n -1)/3 \rfloor +1$, which yields the claimed bounds.

To see that the bounds are attained we consider the sequence, where $\{f_1,f_2, f_3\}$ is a basis of $G$ with
$\ord(f_1) = \ord(f_2) = 2$ and $\ord(f_3) = 2n$, 
\[0^{2\alpha+1} f^{2\beta+1}_1 f^{2\gamma+1}_2 (f_3 + f_1)^{2\alpha+1} (f_3 + f_2)^{2\beta+1} (f_3 + (f_1 + f_2))^{2\gamma+1}\]
where 
\[ 
(\alpha,\beta, \gamma)=
\begin{cases} 
\left(\frac{n}{3},\frac{n}{3}, \frac{n-3}{3}\right)  & \text{ if } n \equiv 0 \pmod{3}\\  
\left(\frac{n-1}{3},\frac{n-1}{3}, \frac{n-1}{3}\right) & \text{ if } n \equiv 1 \pmod{3}\\  
\left(\frac{n+1}{3},\frac{n-2}{3}, \frac{n-2}{3}\right) & \text{ if } n \equiv 2 \pmod{3}
 \end{cases}\] 
which is of the form given in Theorem \ref{Inverse_EGZ} as the sum of  $f_1^{2\alpha+1}  f_2^{2\beta+1} (f_1 + f_2)^{2\gamma+1}$ is $0$.
\end{proof}

An interesting aspect is that for $C_2 \oplus C_2 \oplus C_{2n}$ there are extremal sequences of height significantly below $\exp(G)/2$, 
contrary to all results established so far  (see the discussion in \cite{Wolfgang}, in particular Corollary 3.3 there, where the still stronger conjecture that the height is always $\exp(G)-1$ was refuted).
This seems noteworthy as there is a well-known  technical result (see \cite[Proposition 2.7]{Gao03} or also \cite[Lemma 4.3]{Wolfgang}) that implies  that if $\mathsf{h}(S)\ge \lfloor (\exp(G)-1)/2 \rfloor$ for each sequence $S$ over $G$ of length $\mathsf{s}(G)-1$ without  zero-sum subsequence of length $\exp(G)$, then Gao's conjecture $\mathsf{s}(G)= \eta(G) + \exp(G)-1$ that we recalled in the introduction holds true for the  group $G$. Thus, it seems interesting that despite the relatively low height of some extremal sequences over $C_2  \oplus C_2 \oplus C_{2n}$ Gao's conjecture still holds. 
An explanation for this can be given by a more detailed analysis of the structure of extremal sequences and a recent refinement of the above mentioned technical result (see Lemma \ref{lem_exp-1} below). 

\begin{corollary}
Let $G \simeq C_2 \oplus C_2 \oplus C_{2n}$, where $n \ge 2$. 
For every sequence $S$ over $G$ of length $|S|=\mathsf{s}(G)-1$ without zero-sum subsequence of length $\exp(G)$ there is some $f\in G$ such that $-f + S = CT$ where $T$ is a sequence over $G$ of length $|T|=\eta(G)-1$ without short zero-sum subsequence and 
$C$ is a sequence of length $\exp(G)-1$ of the form $0^{2u + 1}f_1^{2v}f_2^{2w}$ with nonnegative integers $u,v,w$,
and $f_1,f_2\in G$ with $\ord(f_1) = \ord(f_2) = 2$.
\end{corollary}
\begin{proof}
This follows rather directly by comparing the sequences given in Theorems \ref{Inverse_eta} and \ref{Inverse_EGZ}.
For a sequence $S$ of the form given in point $(\mathsf{s} 1)$ the sequence $T$ is of the form given in $(\eta 1)$ and likewise for the other points; for the sequences in $(\mathsf{s} 1)$ and $(\mathsf{s} 2)$ only more special forms of sequences $C$ arise. 
\end{proof}

We note that the sequences $(f+C)$, with $C$ as in the corollary above, have the property $j f \in \Sigma_{j}(f+C)$ for each $j \le |f+C|$ 
thus the following lemma (see \cite[Lemma 4.4]{GirardSchmid1}) is applicable.

\begin{lemma}
\label{lem_exp-1}
Let $G$ be a finite abelian group. Let $S$ be a sequence over $G$ of length $\eta(G) +  \exp(G)-1$. 
Let $C' \mid S$ be a subsequence such that there exists some $f \in G$ with $j f \in \Sigma_{j}(C')$ for each $j \le |C'|$. 
If $|C'|  \ge \lfloor (\exp(G ) - 1)/2 \rfloor$, then $S$ has a zero-sum subsequence of length $\exp(G)$.
\end{lemma}

After this discussion we now turn to the proof of the theorem itself.

\begin{proof}[Proof of Theorem \ref{Inverse_EGZ}]
Let $G \simeq C_2 \oplus C_2 \oplus C_{2n}$, where $n \geq 2$. 
As in the previous proof, we use the inductive method with 
\[
H \hookrightarrow G \overset{\pi} \rightarrow G/H
\]
where $H$ is the unique cyclic subgroup of order $n$ such that $G/H \simeq C_2^3$. 

Let $S$ be a sequence over $G$ such that $\left|S\right| = 4n+2$ that contains no zero-sum subsequence of length $2n$. 
Since $\left|S\right| = 2(2n-4)+ 10 \ge \mathsf{s}(G/H)+2(2n-4)$, there exist $2n-3$ disjoint subsequences $S_{1},\dots,S_{2n-3}$ of $ S$ such that $\left|S_i\right| = 2$ and $\sigma\left(S_i\right) \in H$ for all $i \in  [1,2n-3]$. 
In particular, $S=S_1 \cdots S_{2n-3} S_0$ where $S_0$ has length $|S_0|=8$.

\medskip
On the one hand, all elements of $\pi(S_0)$ have multiplicity at most three and at most one of them has multiplicity at least two, otherwise we could extract from $S_0$ two disjoint subsequences $S_{2n-2}$ and $S_{2n-1}$ satisfying $|S_{2n-2}|=|S_{2n-1}|=2$ and $\sigma\left(S_{2n-2}\right), \sigma\left(S_{2n-1}\right) \in H$ so that $S$ would contain a zero-sum subsequence of length $2n$, a contradiction. In particular, we have $|\supp(\pi(S_0))| \ge 6$. 

\medskip
On the other hand, $T=\prod^{2n-3}_{i=1} \sigma(S_i)$ is a sequence over $H$ of length $|T|=2n-3$, containing no zero-sum subsequence of length $n$, otherwise $S$ would contain a zero-sum subsequence of length $2n$, a contradiction.

\medskip 
By Theorem \ref{Inverse_EGZ_cyc} we get that $|H \setminus \left(-\Sigma_{n-2}\left(T\right)\right)|\le 1$.
Since $|\supp(\pi(S_0))| \ge 6$ it follows by Lemma \ref{sum_=4} that one can extract from $S_0$ a subsequence $U \mid S_0$ such that $\left|U\right| = 4$ and $\sigma\left(U\right) \in H$. Since $S$ contains no zero-sum subsequence of length $2n$, it follows that  $\sigma\left(U\right)$ is an element of $H \setminus -\left( \Sigma_{n-2}\left(T\right)\right)$. Since we saw that $|H \setminus -\left(\Sigma_{n-2}\left(T\right)\right)|\le 1$, it follows that this element is uniquely determined and $|H \setminus \left(- \Sigma_{n-2}\left(T\right)\right)|= 1$.  

\medskip 
By Theorem \ref{Inverse_EGZ_cyc} we get that  there exist two elements $g,h \in H$ satisfying $\ord(h-g)=n$  such that, by relabeling the $S_i$ for $i \in [1,2n-3]$ if necessary, $S$ can be decomposed as $S_1 \cdots S_{2n-3} S_0$ where $\sigma(S_i)=g$ for all $i \in [1,n-1]$ and $\sigma(S_i)=h$ for all $i \in [n, 2n-3]$.
In particular,  $H \setminus \left(- \Sigma_{n-2}\left(T\right)\right) =\{g+h\}$.
We now distinguish cases according to the cardinality of $|\supp(\pi(S_0)|$.

\medskip
\textbf{Case 1.} $|\supp(\pi(S_0))| = 8$. In this case, let $x$ be any element of $S_0$ and $S'_0=S_0x^{-1}$.
Note that $S'_0$ is a subsequence of $S_0$ of length $|S'_0|=7$ and that $\pi(S'_0)$ consists of seven distinct elements of $G/H$. 
Now, let $q,r$ be two distinct elements of $S'_0$ and decompose $S'_0r^{-1}$ in such a way that $S'_0r^{-1}=qstuvw$, 
where $\pi(q)+\pi(s)=\pi(t)+\pi(u)=\pi(v)+\pi(w)=\pi(r)+\pi(x)$. 
It is easily seen that $\pi(s)+\pi(t)+\pi(v)$ and $\pi(s)+\pi(t)+\pi(w)$ are two distinct elements of the set 
$(G/H) \setminus \pi(S'_0r^{-1})=\{\pi(r),\pi(x)\}$,
so that one of them, say $\pi(s)+\pi(t)+\pi(w)$, is equal to $\pi(r)$.
It follows that $U_1=sqtu, U_2=sqvw, U_3=strw, U_4=suvr$ and $U_5=tuvw$ are five subsequences of $S'_0$ of length $4$ with sum in $ H$.

\medskip
If $n=2$, let $y$ be the only element of $H$ satisfying $\ord(y)=n$. 
Since $S$ contains no zero-sum subsequence of length $4$, we have $\sigma(U_i)=y$ for every $i \in [1,5]$.
Therefore, we have $\sigma(U_1)+\sigma(U_2)-\sigma(U_3)-\sigma(U_4)=0$ which is equivalent to $2q=2r$. 
Since the argument applies to any three distinct elements $q,r,x$ of $S_0$, we obtain that $2a \in H $ is constant over all elements $a$ of $S_0$.
Now, $\sigma(U_1)+\sigma(U_2)-\sigma(U_5)=y$ readily gives $0=2(2q)=2s+2q=y$, which is a contradiction.

\medskip
If $n \ge 3$, then $\sigma(U_i) \in H \setminus \left(-\Sigma_{n-2}(T)\right)=\{g+h\}$ for all $i \in [1,5]$.
Therefore, $\sigma(U_i)=g+h$ for all $i \in [1,5]$ which yields
$\sigma(U_1)+\sigma(U_2)-\sigma(U_3)-\sigma(U_4)=0$ which is equivalent to $2q=2r$. 
Since the argument applies to any three distinct elements $q,r,x$ of $S_0$, we obtain that $2a \in H$ is constant over all elements $a$ of $S_0$.
The equality $\sigma(U_1)+\sigma(U_2)-\sigma(U_5)=g+h$ readily gives $4q=2(s+q)=g+h$ so that $4a=g+h$ for any element $a$ of $S_0$.

\medskip
Since $\supp(\pi(S_0))=G/H$, for any $i \in [1,2n-3]$ and any $a_i \mid S_i$, there exists an element $a$ of $S_0$ such that $\pi(a)=\pi(a_i)$.
Setting $T_0=S_0a^{-1}a_i$, $T_i=S_ia^{-1}_ia$ and $T_j=S_j$ for all $j \in [1,2n-3]$ such that $i \neq j$, we obtain a new decomposition of $S$ having the form
\[
S=T_1\cdots T_{2n-3}T_0, \]
where  there exist two elements $g',h' \in H$ satisfying $\ord(h'-g')=n$  such that, by relabeling the $T_i$ for $i \in [1,2n-3]$ if necessary, we have $\sigma(T_i)=g'$ for all $i \in [1,n-1]$ and $\sigma(T_i)=h'$ for all $i \in [n, 2n-3]$.
Using the same argument as above, we obtain that $2x \in H$ is constant over all elements $x$ of $T_0$, and that $4x=g'+h'$ for any $x$ of $T_0$. 
Since at least seven elements of $T_0$ and $S_0$ are equal, it follows that $g'+h' = g+h$, and $4a = 4a_i$. 

\medskip
\textbf{Case 1.1.}  $n$ is odd. Since $4a = 4a_i$, it follows that $2a = 2a_i$, and since  moreover $\pi(a)= \pi(a_i)$, we get that $a = a_i$. 
Since $a_i$ was chosen arbitrarily, it follows that $\supp(S) = \supp(S_0)$. In particular, this implies that $S_i = a_i^2$ for each $i \in [1,2n-3]$; to see this just recall that $\sigma(S_i) \in H$ implies that both elements in $S_i$ have the same image under $\pi$. 

We recall that $2a$ is constant over all elements in $S_0$, and thus in $S$. 
Yet, this would mean that $\sigma(S_i)$ is the same for all $i \in [1,2n-3]$, that is, $g=h$, a contradiction. 

\medskip
\textbf{Case 1.2.} $n$ is even. Since we assumed $n \ge 3$, we have $n \ge 4$ (and this is actually all that we use). 
Applying the replacement above with some $i \in [n,2n-3]$, it can be seen that $h'= h$ (here we use $n \ge 4$, so that the original sequence contains $h$ not only once).
Moreover, it follows that $S_i = a^2$ where $a$ is again the element of $S_0$ such that $\pi(a_i) = \pi(a)$. 
  
In particular, $h = \sigma(S_i)= 2a$. Since we already know that $4a=g+h$, we obtain $g+h=4a=2(2a)=2h$, that is, $g=h$, a contradiction.

\medskip
\textbf{Case 2.} $|\supp(\pi(S_0))| = 7$. The argument is similar to the preceding case. There exists a subsequence $S'_0 \mid S_0$ of length $|S'_0|=7$ such that $\pi(S'_0)$ consists of seven distinct elements of $H$. 
Now, let $q,r$ be two distinct elements of $S'_0$.
We denote by $\gamma$ the only element in $G/H \setminus \{ \pi(S'_0)\}$ and decompose $S'_0r^{-1}$ in such a way that $S'_0r^{-1}=qstuvw$, where $\pi(q)+\pi(s)=\pi(t)+\pi(u)=\pi(v)+\pi(w)=\pi(r)+\gamma$. 
It is easily seen that $\pi(s)+\pi(t)+\pi(v)$ and $\pi(s)+\pi(t)+\pi(w)$ are two distinct elements of the set 
$G/H \setminus \pi(S'_0r^{-1})=\{\pi(r),\gamma\}$,
so that one of them, say $\pi(s)+\pi(t)+\pi(w)$ is equal to $\pi(r)$.
It follows that $U_1=sqtu, U_2=sqvw, U_3=strw, U_4=suvr$ and $U_5=tuvw$ are five subsequences of $S'_0$ of length $4$ with sum in $ H$.

\medskip
If $n=2$, let $y$ be the only element of $H$ satisfying $\ord(y)=n$. 
Since $S$ contains no zero-sum subsequence of length $4$, we have $\sigma(U_i)=y$ for every $i \in [1,5]$.
Therefore, we have $\sigma(U_1)+\sigma(U_2)-\sigma(U_3)-\sigma(U_4)=0$ which is equivalent to $2q=2r$. 
Since the argument applies to any two distinct elements of $S'_0$, we obtain that $2a \in H$ is constant over all elements $a$ of $S'_0$.
Now, $\sigma(U_1)+\sigma(U_2)-\sigma(U_5)=y$ readily gives $0=2(2q)=2(s+q)=y$, which is a contradiction.

\medskip 
We can thus assume that $n \ge 3$, so that $\sigma(U_i) \in H \setminus \left(-\Sigma_{n-2}(T)\right)=\{g+h\}$ for all $i \in [1,5]$.
Therefore, $\sigma(U_i)=g+h$ for all $i \in [1,5]$ which yields $\sigma(U_1)+\sigma(U_2)-\sigma(U_3)-\sigma(U_4)=0$ which is equivalent to $2q=2r$. 
Since the argument applies to any two distinct elements of $S'_0$, we obtain that $2a \in H$ is constant over all elements $a$ of $S'_0$.
The equality $\sigma(U_1)+\sigma(U_2)-\sigma(U_5)=g+h$ readily gives $4q=2(s+q)=g+h$ so that $4a=g+h$ for any element $a$ of $S'_0$.

\medskip
Now, let $a$ be an element of $S'_0$ and let $V_1$ be a subsequence of $S_0$ such that $|V_1|=6$ and $\sigma(V_1) \in H$; such a sequence exists as $\pi(S_0)$, not being  squarefree, has a zero-sum subsequence of length $2$, and the remaining squarefree sequence of length $6$ has a zero-sum subsequence of length $4$ by Lemma \ref{sum_=4}.  
If $n=3$, then $\sigma(V_1) \in H \setminus \{0\} = \{2h+g, h+2g\}$.
If $n \ge 4$, then $\sigma(V_1) \in H \setminus \left(- \Sigma_{n-3}\left(T\right)\right) = \{2h+g, h+2g\}$ also.
In all cases, we obtain either $2(2h+g)=2\sigma(V_1)=6(2a)=3(4a)=3(g+h)$ or $2(h+2g)=2\sigma(V_1)=6(2a)=3(4a)=3(g+h)$ which both imply $g=h$, a contradiction.

\medskip

\textbf{Case 3.} $\supp(\pi(S_0))=6$. 
We get that $\pi(S_0)$ consists of one element repeated three times and five other elements; recall that at most one element of  $\pi(S_0)$ has multiplicity greater than one. 
Thus $S_0$ contains a subsequence $S_{2n-2} \mid S_0$ such that $\left|S_{2n-2}\right| = 2$ and $\sigma\left(S_{2n-2}\right) \in H$. 
It follows that, whatever $n \ge 2$ we consider, there exist two elements $g,h \in H$ satisfying $\ord(h-g)=n$ such that $S$ can be decomposed as $S_1 \cdots S_{2n-2} S'_0$ where $\sigma(S_i)=g$ for all $i \in [1,n-1]$, $\sigma(S_i)=h$ for all $i \in [n,2n-2]$, and $S'_0$ has length $|S'_0|=6$; furthermore, $\supp(\pi(S_1 \cdots S_{2n-2})) \cap \supp(\pi(S'_0)) \neq \emptyset$.

\medskip
Now, let us prove that, after a translation by an element of $G$, the sequence $S$ can be decomposed as $0 T_1 \cdots T_{n-1} V$ where $\left|T_i\right| = 2$ and $\sigma\left(T_i\right)=0$ for all $i \in [1,n-1]$, and where $V$, which has length $|V|=2n+3=\eta(G)-1$, contains no short zero-sum subsequence. 

\medskip
Since $\supp (\pi(S_1 \cdots S_{2n-2})) \cap \supp (\pi(S'_0) ) \neq \emptyset$, there exist an element $f$ of $S'_0$ and an element $a_i$ of some $S_i$, where $i \in [1,2n-2]$, such that $\pi(f)=\pi(a_i)$.
If $i \in [1,n-1]$, then replacing $S_i$ by $f a_i$ or $f (g-a_i)$, we easily infer that $f=a_i=g-a_i$. 
In particular, $2f=g$, so that 
\[
 -f+S = 0 (-f+S_1) \cdots  (-f+S_{n-1}) V
\]
where $V$ has length $|V|=4n+2-2(n-1)-1=2n+3=\eta(G)-1$. We set $T_i = -f+S_i$ for each $i \in [1,n-1]$. Since for every $t \in [1,2n-1]$, the sequence $0 T_1 \cdots T_{n-1}$  contains a zero-sum subsequence of length $t$, it follows that $V$ has no short zero-sum subsequence; 
otherwise $-f+S$ and thus $S$ would have a zero-sum subsequence of length $2n$. 

If $i \in [n,2n-2]$, the argument is analogous, just replacing $g$ by $h$ and considering the sequence $S_n$, \dots, $S_{2n-2}$.

\medskip
We now know that for every sequence $S$ over $G$ such that $|S| = 4n+2$ and containing no zero-sum subsequence of length $2n$, there is an $f \in G$ such that $-f +S$  can be decomposed as $0 T_1 \cdots T_{n-1} V$ where $|T_i| = 2$ and $\sigma(T_i)=0$ for all $i \in [1,n-1]$, and where $V$, which has length $|V|=2n+3=\eta(G)-1$, contains no short zero-sum subsequence.

\medskip
We recall that $S$ has a zero-sum subsequence of length $2n$ if and only if $-f+S$ has a zero-sum subsequence of length $2n$. Thus, $-f+S$ has no zero-sum subsequence of length $2n$. To simplify notation, we assume without loss that $f =0$. 

\medskip
On the one hand, note that $\Sigma_{[1,2n-1]}(V) = G \setminus \{0\}$. 
On the other hand, we assert that: if $x \mid T_i$ for some $i \in [1,n-1]$, then $x \notin -\Sigma_{[2,2n-1]}(V)$. To see this note that otherwise $V$ would contain a subsequence $V' \mid V$ of length $|V'|=k \in [2,2n-1]$ such that $x=-\sigma(V')$. 
If $k$ were odd, then the subsequence $S'=xV' \prod_{j \in J} T_j \mid S$, where $J$ is any subset of $[1,n-1] \setminus \{i\}$ satisfying $|J|=n-(k+1)/2 \in [0,n-2]$ would have length $|S'|=2n$ and sum zero, a contradiction.
If $k$ were even, then the subsequence $S'=0xV' \prod_{j \in J} T_j \mid S$, where $J$ is any subset of $[1,n-1] \setminus \{i\}$ satisfying $|J|=n-(k+2)/2 \in [0,n-2]$ would have length $|S'|=2n$ and sum zero, a contradiction.
Therefore, any $x \in \supp(T_1 \cdots T_{n-1})$ is either zero or an element of $G \setminus \{0\} \setminus (-\Sigma_{[2,2n-1]}(V)) = -\Sigma_{[1,2n-1]}(V) \setminus (-\Sigma_{[2,2n-1]}(V))=-\supp(V)$.

Since $x \mid T_i$ if and only if $-x \mid T_i$, we obtain 
\[
\supp(T_1 \cdots T_{n-1}) \subset \{0\} \cup  \bigl(\supp(V) \cap (-\supp(V)) \bigr).\] 
Finally, using Theorem \ref{Inverse_eta}, we know that there exists a basis $\{f_1,f_2, f_3\}$ of $G$, where $\ord(f_1) = \ord(f_2) = 2$ and $\ord(f_3) = 2n$, such that $V$ has one out of three possible forms. 

\medskip
If $V$ is of type $(\eta 1)$, then
\[
\supp(T_1 \cdots T_{n-1}) \subset \{0\} \cup \bigl(\supp(V) \cap (-\supp(V))\bigr)=\{0,f_2\},
\]
so that we have 
\[
S=0^{2\alpha+1} f^{2n-2\alpha-2}_2 f^{2n-1-2v}_3 (f_3+f_2)^{2v+1}(af_3+f_1)\bigl((1-a)f_3+f_2+f_1\bigr)
\]
with $\alpha,\beta  \in [ 0, n-1 ]$ and $a \in [ 2, n-1 ]$.
It remains to check that such a sequence contains no zero-sum subsequence of length $2n$.
Since $\Sigma(0^{2\alpha+1} f^{2n-2\alpha-2}_2) \subset \{0,f_2\}$,  it suffices to check that  $0 \notin \Sigma_{[1,2n]}\bigl(f^{2n-1-2v}_3 (f_3+f_2)^{2v+1}(af_3+f_1)\bigl((1-a)f_3+f_2+f_1\bigl)\bigl)$ and  that $f_2 \notin \Sigma_{[1,2n-1]}(f^{2n-1-2v}_3 (f_3+f_2)^{2v+1}(af_3+f_1)\bigl((1-a)f_3+f_2+f_1\bigr)\bigr)$. The former is a consequence of Theorem \ref{Inverse_eta} and the latter follows by noting that if a subsequence with sum $f_2$ contains $(af_3+f_1)$ or $\bigl((1-a)f_3+f_2+f_1\bigr)$, then it contains both, so that we need at least $\ord(f_3)=2n$ elements in this subsequence. 

\medskip
If $V$ is of type $(\eta 2)$, then
\[
\supp(T_1 \cdots T_{n-1}) \subset \{0\} \cup \bigl(\supp(V) \cap (-\supp(V))\bigr)=\{0\},
\]
so that, up to translation by an element of $G$, we have 
\[
S=0^{2n-1}f^{2n-1}_3 (af_3 + f_2)\bigl((1 - a)f_3 + f_2\bigr)(bf_3 + f_1)\bigl((1 - b)f_3 + f_1\bigr)
\] 
with $a, b \in [ 2, n - 1 ]$ and $a \ge b$.
Such a sequence contains no zero-sum subsequence of length $2n$ indeed.

\medskip
If $V$ is of type $(\eta 3)$, then
\[
\supp(T_1 \cdots T_{n-1}) \subset \{0\} \cup \bigl(\supp(V) \cap (-\supp(V))\bigr)=\{ 0, f_1, f_2 \},
\]
so that, up to translation by an element of $G$, we have
\[
S=0^{2\alpha+1} f_1^{2\beta+1} f_2^{2\gamma+1} \prod^{2n+1}_{i=1}(f_3 + d_i)
\] 
where $\alpha,\beta,\gamma \in [0,n-1]$ are such that $\alpha+\beta+\gamma=n-1$, and $S' = \prod^{2n+1}_{i=1} d_i \in \mathcal{F}(\langle f_1, f_2 \rangle)$ and $\sigma(S') \notin \supp(S')$.
It remains to check that such a sequence contains no zero-sum subsequence of length $2n$.
Since $\Sigma(0^{2\alpha+1} f_1^{2\beta+1} f_2^{2\gamma+1}) = \{0,f_1,f_2\}$,  it suffices to check that  $0 \notin \Sigma_{[1,2n]}(\prod^{2n+1}_{i=1}(f_3 + d_i))$ and  that $f_1,f_2 \notin \Sigma_{[1,2n-1]}(\prod^{2n+1}_{i=1}(f_3 + d_i))$.
The former is a consequence of Theorem \ref{Inverse_eta} and the latter follows by noting that we need at least $\ord(f_3)=2n$ elements in a subsequence of  $\prod^{2n+1}_{i=1}(f_3 + d_i)$ whose sum is in $\langle f_1 , f_2 \rangle$. 
\end{proof}

\section*{Acknowledgments}
The authors thank the referee for a careful reading and useful remarks that helped to improve the paper.

\end{document}